\documentclass{amsart}
\usepackage{graphicx}
\newtheorem{thm}{Theorem}[section]
\newtheorem{cor}[thm]{Corollary}
\newtheorem{lem}[thm]{Lemma}
\newtheorem{prop}[thm]{Proposition}
\theoremstyle{definition}
\newtheorem{defn}[thm]{Definition}
\theoremstyle{example}
\newtheorem{exam}[thm]{Example}
\theoremstyle{remark}

\numberwithin{equation}{section}

\begin{document}

\title[Cross-Gram Matrix associated to two sequences in Hilbert spaces]
{Cross-Gram Matrix associated to two sequences in Hilbert spaces}

\author{E. Osgooei, A. Rahimi}
\address{Department of sciences\\
Urmia University of Technology\\ Urmia, Iran.}
\email{osgooei@yahoo.com; e.osgooei@uut.ac.ir}
\address{Department of Mathematics\\ University of Maragheh\\ Maragheh,
Iran.} \email{rahimi@maragheh.ac.ir; asgharrahimi@yahoo.com}

\dedicatory{}



\begin{abstract}
The conditions for sequences $\{f_{k}\}_{k=1}^{\infty}$ and
$\{g_{k}\}_{k=1}^{\infty}$ being Bessel sequences, frames or Riesz
bases, can be expressed in terms of the so-called cross-Gram matrix.
In this paper we investigate the cross-Gram operator, $G$,
associated to the sequence $\{\langle f_{k}, g_{j}\rangle\}_{j,
k=1}^{\infty}$ and sufficient and necessary conditions for
boundedness, invertibility, compactness and positivity of this
operator are determined depending on the associated sequences. We
show that invertibility of $G$ is not possible when the associated
sequences are frames but not Riesz Bases or at most one of them is
Riesz basis. In the special case we prove that $G$ is a positive
operator when $\{g_{k}\}_{k=1}^{\infty}$ is the canonical dual of
$\{f_{k}\}_{k=1}^{\infty}$.
\end{abstract}

\maketitle
\section{INTRODUCTION}
The fundamental operators in frame theory are the synthesis,
analysis and frame operators associated with a given frame. The
ability of combining these operators to make a sensitive operator is
indeed essential in frame theory and its applications.
Time-invariant filters', i. e. convolution operators, are used
frequently in applications. These operators can be called Fourier
multipliers \cite{9, 16}. In the last decade Gabor filters which are
beneficial tools to perform time-variant filters have very strong
applications in psychoacoustics \cite{8}, computational auditory
scene analysis \cite{30}, and seismic data analysis \cite{22}. For
more information on these operators we refer to \cite{12, 10, 11}.
\\In this paper for given two Bessel sequences $\{f_{k}\}_{k=1}^{\infty}$ and
$\{g_{k}\}_{k=1}^{\infty}$, the synthesis operator of the sequence
$\{f_{k}\}_{k=1}^{\infty}$ with the analysis operator of the
sequence $\{g_{k}\}_{k=1}^{\infty}$ is composed and a fundamental
operator is generated. This operator is called the cross-Gram
operator associated with the sequence $\{\langle f_{k},
g_{j}\rangle\}_{j, k=1}^{\infty}$ \cite{B, p}. This paper concerns
this question that when can sequences $\{f_{k}\}_{k=1}^{\infty}$ and
$\{g_{k}\}_{k=1}^{\infty}$ in a Hilbert space $H$, generate a
cross-Gram operator with the properties of boundedness,
invertibility and positivity. Vise versa if the cross Gram-operator
has the above properties what can be expected of the sequences
$\{f_{k}\}_{k=1}^{\infty}$ and $\{g_{k}\}_{k=1}^{\infty}$.
\\Let $H$ be a complex Hilbert space. A frame for $H$ is a sequence
$\{f_{k}\}_{k=1}^{\infty}\subset H $ such that there are positive
constants $A$ and $B$ satisfying
\begin{equation}\label{9999}
A\|f\|^{2}\leq\sum_{k=1}^{\infty}|\langle f, f_{k}\rangle|^{2}\leq
B\|f\|^{2}, \ \ f\in H.
\end{equation}
The constants $A$ and $B$ are called lower and upper frame bounds,
respectively. We call $\{f_{k}\}_{k=1}^{\infty} $ a Bessel sequence
with bound $B$, if we have only the second inequality in
(\ref{9999}). Associated with each Bessel sequence
$\{f_{k}\}_{k=1}^{\infty}$ we have three linear and bounded
operators, the synthesis operator:
$$T:\ell^{2}(\mathbb{N})\rightarrow H,\ \
T(\{c_{k}\}_{k=1}^{\infty})=\sum_{k=1}^{\infty}c_{k}f_{k},$$ the
analysis operator which is defined by:
$$T^{*}:H\rightarrow\ell^{2}(\mathbb{N});\ \ T^{*}f=\{\langle f,
f_{k}\rangle\}_{k=1}^{\infty},$$ and the frame operator:
$$S:H\rightarrow H;\ \ Sf=TT^{*}f=\sum_{k=1}^{\infty}\langle f, f_{k}\rangle
f_{k}.$$ If $\{f_{k}\}_{k=1}^{\infty}$ is a Bessel sequence, we can
compose the synthesis operator $T$ and its adjoint $T^{*}$ to obtain
the bounded operator
$$T^{*}T:\ell^{2}(\mathbb{N})\rightarrow\ell^{2}(\mathbb{N});\ \
T^{*}T\{c_{k}\}_{k=1}^{\infty}=\{\langle\sum_{\ell=1}^{\infty}c_{\ell}f_{\ell},
f_{k}\rangle\}_{k=1}^{\infty}.$$
Therefore the matrix
representation of $T^{*}T$ is as follows:
$$T^{*}T=\{\langle f_{k}, f_{j}\rangle\}_{j, k=1}^{\infty}.$$
The matrix $\{\langle f_{k}, f_{j}\rangle\}_{j, k=1}^{\infty}$ is
called the matrix associated with $\{f_{k}\}_{k=1}^{\infty}$ or Gram
matrix and it defines a bounded operator on $\ell^{2}(\mathbb{N})$
when $\{f_{k}\}_{k=1}^{\infty}$ is a Bessel sequence.
\\In order to recognize that a sequence $\{f_{k}\}_{k=1}^{\infty}$ is a Bessel
sequence or frame we need to check (\ref{9999}) for all $f\in H$.
But in practice this is not always so easy. The following two
results give us a practical method to diagnose Bessel sequences or
frames by the concept of Gram matrix or in other words just by
calculating $\{\langle f_{k}, f_{j}\rangle\}_{j, k=1}^{\infty}$.
\begin{lem}\label{1111}\cite{7}
Suppose that $\{f_{k}\}_{k=1}^{\infty}\subseteq H$. Then the
following statements are equivalent:
\\(1) $\{f_{k}\}_{k=1}^{\infty}$ is a Bessel sequence with bound $B$.
\\(2) The Gram matrix associated with $\{f_{k}\}_{k=1}^{\infty}$ defines a bounded
operator on $\ell^{2}(\mathbb{N})$ with norm at most $B$.
\end{lem}
\begin{defn}
A Riesz basis $\{f_{k}\}_{k=1}^{\infty}$ for $H$ is a family of the
form $\{Ue_{k}\}_{k=1}^{\infty}$, where $\{e_{k}\}_{k=1}^{\infty}$
is an orthonormal basis for $H$ and $U:H\rightarrow H$ is a bounded
bijective operator.
\end{defn}
\begin{prop}\label{ch}\cite{7}
A sequence $\{f_{k}\}_{k=1}^{\infty}$ is a Riesz basis for $H$ if
and only if it is an unconditional basis for $H$ and
$$0<\inf\|f_{k}\|\leq\sup\|f_{k}\|<\infty.$$
\end{prop}
\begin{thm}\label{33333}\cite{7}
Suppose that $\{f_{k}\}_{k=1}^{\infty}\subseteq H$. Then the
following conditions are equivalent:
\\(1) $\{f_{k}\}_{k=1}^{\infty}$ is a Riesz basis for $H$.
\\(2) $\{f_{k}\}_{k=1}^{\infty}$ is complete and its Gram matrix $\{\langle f_{k},
f_{j}\rangle\}_{j, k=1}^{\infty}$ defines a bounded, invertible
operator on $\ell^{2}(\mathbb{N})$.
\end{thm}
\section{Cross-Gram matrix}
If $\{f_{k}\}_{k=1}^{\infty}$ and $\{g_{k}\}_{k=1}^{\infty}$ are
Bessel sequences, we compose the synthesis operator of the sequence
$\{f_{k}\}_{k=1}^{\infty}$, $T_{f_{k}}$, and the analysis operator
of the sequence $\{g_{k}\}_{k=1}^{\infty}$, $T^{*}_{g_{k}}$, to
obtain a bounded operator on $\ell^{2}(\mathbb{N})$
$$T^{*}_{g_{k}}T_{f_{k}}:\ell^{2}(\mathbb{N})\rightarrow\ell^{2}(\mathbb{N});\
\
T^{*}_{g_{k}}T_{f_{k}}\{c_{k}\}_{k=1}^{\infty}=\{\langle\sum_{\ell=1}^{\infty}c_{\ell}f_{\ell},
g_{k}\rangle\}_{k=1}^{\infty}.$$ This operator,
$G=T^{*}_{g_{k}}T_{f_{k}}$, is called the cross-Gram operator
associated to $\{\langle f_{k}, g_{j}\rangle\}_{j, k=1}^{\infty}$
\cite{B, p}.
\\If $\{e_{k}\}_{k=1}^{\infty}$ is the canonical orthonormal basis for
$\ell^{2}(\mathbb{N})$, the jk-th entry in the matrix representation
for $T^{*}_{g_{k}}T_{f_{k}}$ is
$$\langle T^{*}_{g_{k}}T_{f_{k}}e_{k},
e_{j}\rangle=\langle T_{f_{k}}e_{k}, T_{g_{k}}e_{j}\rangle=\langle
f_{k}, g_{j}\rangle.$$ Therefore the matrix representation of
$T^{*}_{g_{k}}T_{f_{k}}$ is as follows:
$$T^{*}_{g_{k}}T_{f_{k}}=\{\langle f_{k}, g_{j}\rangle\}_{j,
k=1}^{\infty}.$$ The matrix $\{\langle f_{k}, g_{j}\rangle\}_{j,
k=1}^{\infty}$ is called the cross-Gram matrix associated to
$\{f_{k}\}_{k=1}^{\infty}$ and $\{g_{k}\}_{k=1}^{\infty}$ \cite{B,
p}. In the special case that the sequences
$\{f_{k}\}_{k=1}^{\infty}$ and $\{g_{k}\}_{k=1}^{\infty}$ are
biorthogonal, the cross-Gram matrix is the identity matrix.
\\The above discission shows that the cross-Gram matrix is bounded
above if $\{f_{k}\}_{k=1}^{\infty}$ and $\{g_{k}\}_{k=1}^{\infty}$
are Bessel sequences.
The following example shows that the inverse of the above assertion
is not valid, in other words, the cross-Gram matrix associated to
two sequences can be well-defined and bounded in the case that one
of the sequences is not Bessel.
\begin{exam}\label{ka}
Suppose that $\{e_{k}\}_{k=1}^{\infty}$ is the orthonormal basis for
a Hilbert space $H$. Consider
$\{f_{k}\}_{k=1}^{\infty}=\{\frac{1}{k}e_{k}\}_{k=1}^{\infty}$ and
$\{g_{k}\}_{k=1}^{\infty}=\{ke_{k}\}_{k=1}^{\infty}$. A simple
calculation shows that the cross-Gram matrix associated to these
sequences is the identity matrix, but $\{g_{k}\}_{k=1}^{\infty}$ is
not a Bessel sequence.
\end{exam}
\begin{defn}
Let $U$ be an operator on a Hilbert space $H$, and suppose that $E$
is an orthonormal basis for $H$. We say that $U$ is a
Hilbert-Schmidt operator if
$$\|U\|_{2}=(\sum_{x\in E}\|Ux\|^{2})^{\frac{1}{2}}<\infty.$$
\end{defn}
\begin{thm}
Suppose that $\{f_{k}\}_{k=1}^{\infty}$ and
$\{g_{k}\}_{k=1}^{\infty}$ are sequences in $H$ and
$\{g_{k}\}_{k=1}^{\infty}$ is a Bessel sequence with bound $B'$.
Assume that there exists $M>0$ such that
$\sum_{k=1}^{\infty}\|f_{k}\|^{2}\leq M.$ Then the cross-Gram
operator associated to $\{\langle f_{k}, g_{j}\rangle\}_{j,
k=1}^{\infty}$ is a well-defined, bounded and compact operator.
\end{thm}
\begin{proof}
Suppose that $G=\{\langle f_{k}, g_{j}\rangle\}_{j, k=1}^{\infty}$.
For a given sequence
$\{c_{k}\}_{k=1}^{\infty}\in\ell^{2}(\mathbb{N})$ we have
\begin{eqnarray*}
\|G\{c_{k}\}_{k=1}^{\infty}\|^{2}&=&\sum_{j=1}^{\infty}|\sum_{k=1}^{\infty}c_{k}\langle
f_{k},
g_{j}\rangle|^{2}\\&\leq&\sum_{j=1}^{\infty}\sum_{k=1}^{\infty}|c_{k}|^{2}\sum_{k=1}^{\infty}|\langle
f_{k},
g_{j}\rangle|^{2}\\&=&\sum_{k=1}^{\infty}|c_{k}|^{2}\sum_{k=1}^{\infty}\sum_{j=1}^{\infty}|\langle
f_{k}, g_{j}\rangle|^{2}\\&\leq&
B'\sum_{k=1}^{\infty}|c_{k}|^{2}\sum_{k=1}^{\infty}\|f_{k}\|^{2}\\&\leq&
B'M\sum_{k=1}^{\infty}|c_{k}|^{2}.
\end{eqnarray*}
By above assertion,
$G\{c_{k}\}_{k=1}^{\infty}\in\ell^{2}(\mathbb{N})$ and therefore $G$
is well-defined and bounded.
\\ Now suppose that $\{e_{k}\}_{k=1}^{\infty}$ is the canonical
orthonormal basis for $\ell^{2}(\mathbb{N})$. Then
\begin{eqnarray*}
(\sum_{k=1}^{\infty}\|G(e_{k})\|^{2})^{\frac{1}{2}}&=&(\sum_{k=1}^{\infty}\sum_{j=1}^{\infty}|\langle
f_{k}, g_{j}\rangle|^{2})^{\frac{1}{2}}\\
&\leq&\sqrt{B'}(\sum_{k=1}^{\infty}\|f_{k}\|^{2})^{\frac{1}{2}}\leq\sqrt{B'M}.
\end{eqnarray*}
Therefore $G$ is a Hilbert-Schmidt operator and so is
compact\cite{ped}.
\end{proof}
\begin{exam}
Let $\{e_{k}\}_{k=1}^{\infty}$ be an orthonormal basis for $H$.
Consider $\{f_{k}\}_{k=1}^{\infty}=\{e_{1},\frac{1}{2}e_{2},
\frac{1}{3}e_{3}, \frac{1}{4}e_{4},...\}$ and
$\{g_{k}\}_{k=1}^{\infty}=\{\frac{1}{2}e_{1}, e_{2},
\frac{1}{2^{2}}e_{1}, e_{3},...\}.$ Suppose that $G$ is the
cross-Gram operator associated to $\{\langle f_{k},
g_{j}\rangle\}_{j, k=1}^{\infty}$. A simple calculation shows that
$$\sum_{k=1}^{\infty}\|G(e_{k})\|^{2}=\sum_{k=1}^{\infty}\frac{1}{2^{2k}}+\sum_{k=1}^{\infty}\frac{1}{k^{2}}.$$
Therefore $G$ is a Hilbert-Schmidt operator and so is compact.
\end{exam}
If $\sup_{k}\|f_{k}\|<\infty$ (resp. $\inf_{k}\|f_{k}\|>0$), the
sequence $\{f_{k}\}_{k=1}^{\infty}$ will be called norm-bounded
above or NBA, (resp. norm-bounded below or NBB).
\begin{thm}
Let $\{f_{k}\}_{k=1}^{\infty}$ and $\{g_{k}\}_{k=1}^{\infty}$ be
sequences for $H$ and $G$ be the cross-Gram operator associated to
$\{\langle f_{k}, g_{j}\rangle\}_{j, k=1}^{\infty}$. Then the
following statements are satisfied:
\\(i) Assume that $G$ is well-defined and bounded above and $\{g_{k}\}_{k=1}^{\infty}$ is a frame with lower bound $A$.
Then $\{f_{k}\}_{k=1}^{\infty}$ is norm bounded above.
\\(ii) Assume that $G$ is well-defined and bounded below and $\{g_{k}\}_{k=1}^{\infty}$ is a Bessel sequence with upper bound
$B$. Then $\{f_{k}\}_{k=1}^{\infty}$ is norm bounded below.
\\(iii) Assume that $G$ is well-defined and bounded above and $\{f_{k}\}_{k=1}^{\infty}$ is an orthonormal basis for $H$.
Then $\{g_{k}\}_{k=1}^{\infty}$ is a Bessel sequence for $H$.
\end{thm}
\begin{proof}
(i) Since $G$ is well-defined and bounded, there exists a constant
$M>0$ such that
\begin{equation}\label{seen}
\sum_{j=1}^{\infty}|\sum_{k=1}^{\infty}c_{k}\langle f_{k},
g_{j}\rangle|^{2}\leq M\sum_{k=1}^{\infty}|c_{k}|^{2}.
\end{equation}
Via (\ref{seen}) applied to the elements of the canonical
orthonormal basis of $\ell^{2}(\mathbb{N})$, we have
\begin{equation}\label{saw}
\sum_{j=1}^{\infty}|\langle f_{k}, g_{j}\rangle|^{2}\leq M,\ \
k\in\mathbb{N}.
\end{equation}
Since $\{g_{k}\}_{k=1}^{\infty}$ is a frame for $H$, by (\ref{saw})
we have
\begin{equation}
A\|f_{k}\|^{2}\leq\sum_{j=1}^{\infty}|\langle f_{k},
g_{j}\rangle|^{2}\leq M,\ \ \ell\in\mathbb{N}.
\end{equation}
Therefore
$$\|f_{k}\|^{2}\leq \frac{M}{A},\ \ k\in\mathbb{N}.$$
(ii) By assumption there exist $M'>0$, such that
\begin{equation}\label{shiv}
M'\sum_{k=1}^{\infty}|c_{k}|^{2}\leq\sum_{j=1}^{\infty}|\sum_{k=1}^{\infty}c_{k}\langle
f_{k}, g_{j}\rangle|^{2}.
\end{equation}
Similar to above discussion, since $\{g_{k}\}_{k=1}^{\infty}$ is a
Bessel sequence, for each $k\in\mathbb{N}$, via (\ref{shiv}) applied
to the elements of the canonical orthonormal basis of
$\ell^{2}(\mathbb{N})$, we have
\begin{equation}
M' \leq\sum_{j=1}^{\infty}|\langle f_{k}, g_{j}\rangle|^{2}\leq
B\|f_{k}\|^{2}.
\end{equation}
Therefore
$$\|f_{k}\|^{2}\geq\frac{M'}{B},\ \ k\in\mathbb{N}.$$
\\(iii) Since $G$ is well-defined and bounded above there exists $M''>0$,
such that
\begin{eqnarray}\label{maja}
\sum_{j=1}^{\infty}|\sum_{k=1}^{\infty}c_{k}\langle f_{k},
g_{j}\rangle|^{2}\leq M''\sum_{k=1}^{\infty}|c_{k}|^{2}.
\end{eqnarray}
Since $\{f_{k}\}_{k=1}^{\infty}$ is an orthonormal basis for $H$,
there exist a sequence $\{c_{k}\}\in\ell^{2}(\mathbb{N})$, such that
for each $f\in H$ we have
\begin{eqnarray*}
\sum_{j=1}^{\infty}|\langle f,
g_{j}\rangle|^{2}=\sum_{j=1}^{\infty}|\langle\sum_{k=1}^{\infty}c_{k}f_{k},
g_{j}\rangle|^{2}.
\end{eqnarray*}
Now by (\ref{maja}) we have
\begin{eqnarray*}
\sum_{j=1}^{\infty}|\sum_{k=1}^{\infty}c_{k}\langle f_{k},
g_{j}\rangle|^{2}\leq M''\sum_{k=1}^{\infty}|c_{k}|^{2},
\end{eqnarray*}
Therefore
\begin{eqnarray*}
\sum_{j=1}^{\infty}|\langle f, g_{j}\rangle|^{2}\leq
M''\sum_{k=1}^{\infty}|c_{k}|^{2}=M''\|f\|^{2}.
\end{eqnarray*}
\end{proof}
\begin{prop}
Suppose that $\{f_{k}\}_{k=1}^{\infty}$ and
$\{g_{k}\}_{k=1}^{\infty}$ are Bessel sequences and $G$ is the
cross-Gram operator associated to $\{\langle f_{k},
g_{j}\rangle\}_{j, k=1}^{\infty}$. Then the following statements are
satisfied:
\\(i) If $\{f_{k}\}_{k=1}^{\infty}$ is a Riesz basis and $\{g_{k}\}_{k=1}^{\infty}$ is a frame for
$H$, then $G$ is a bounded injective operator.
\\(ii) If $\{f_{k}\}_{k=1}^{\infty}$ is a frame and $\{g_{k}\}_{k=1}^{\infty}$ is a Riesz basis
for $H$, then $G$ is a bounded surjective operator.
\end{prop}
\begin{proof}
Since $\{f_{k}\}_{k=1}^{\infty}$ and $\{g_{k}\}_{k=1}^{\infty}$ are
Bessel sequences, we deduce that $G$ is a well-defined and bounded
operator.
\\(i) Suppose that
$$G\{c_{k}\}_{k=1}^{\infty}=G\{b_{k}\}_{k=1}^{\infty},\ \ \{c_{k}\}_{k=1}^{\infty},
\{b_{k}\}_{k=1}^{\infty}\in\ell^{2}(\mathbb{N}).$$ Then
$T^{*}_{g_{k}}T_{f_{k}}\{c_{k}\}_{k=1}^{\infty}=T^{*}_{g_{k}}T_{f_{k}}\{b_{k}\}_{k=1}^{\infty}.$
Since $\{g_{k}\}_{k=1}^{\infty}$ is a frame for $H$, $T^{*}_{g_{k}}$
is an injective operator and we have
$T_{f_{k}}\{c_{k}\}_{k=1}^{\infty}=T_{f_{k}}\{b_{k}\}_{k=1}^{\infty}.$
Since $\{f_{k}\}_{k=1}^{\infty}$ is a Riesz basis, $T_{f_{k}}$ is
invertible. So $\{c_{k}\}_{k=1}^{\infty}=\{b_{k}\}_{k=1}^{\infty}$
and we get the result.
\\(ii) Since $\{g_{k}\}_{k=1}^{\infty}$ is a Riesz basis,
$T^{*}_ {g_{k}}$ is a bijective operator. Therefore for a given
sequence $\{c_{k}\}_{k=1}^{\infty}\in\ell^{2}(\mathbb{N})$, there
exist $h\in H$ such that $T^{*}_{g_{k}}h=\{c_{k}\}_{k=1}^{\infty}.$
Also by assumption $T_{f_{k}}$ is a surjective operator and there
exists $\{b_{k}\}_{k=1}^{\infty}\in\ell^{2}(\mathbb{N})$ such that
$T_{f_{k}}\{b_{k}\}_{k=1}^{\infty}=h.$ Therefore we have
$T^{*}_{g_{k}}T_{f_{k}}\{b_{k}\}_{k=1}^{\infty}=\{c_{k}\}_{k=1}^{\infty}$
and so $G\{b_{k}\}_{k=1}^{\infty}=\{c_{k}\}_{k=1}^{\infty}.$
\end{proof}
\begin{thm}\label{yani}
Suppose that $\{f_{k}\}_{k=1}^{\infty}$ and
$\{g_{k}\}_{k=1}^{\infty}$ are Riesz bases for $H$. Then
$\{f_{k}\}_{k=1}^{\infty}$ and $\{g_{k}\}_{k=1}^{\infty}$ are
complete and the cross-Gram matrix associated to
$\{f_{k}\}_{k=1}^{\infty}$ and $\{g_{k}\}_{k=1}^{\infty}$ defines a
bounded invertible operator on $\ell^{2}(\mathbb{N})$.
\end{thm}
\begin{proof}
Since $\{f_{k}\}_{k=1}^{\infty}$ and $\{g_{k}\}_{k=1}^{\infty}$ are
Riesz bases, there exist bijective operators $U$ and $W$ such that
$\{f_{k}\}_{k=1}^{\infty}=\{Ue_{k}\}_{k=1}^{\infty}$ and
$\{g_{k}\}_{k=1}^{\infty}=\{We_{k}\}_{k=1}^{\infty}$, where
$\{e_{k}\}_{k=1}^{\infty}$ is an orthonormal basis of $H$. For every
$k, j\in\mathbb{N}$ we have
$$\langle f_{k}, g_{j}\rangle=\langle Ue_{k}, We_{j}\rangle=\langle
W^{*}Ue_{k}, e_{j}\rangle.$$ i. e., the cross-Gram matrix associated
to $\{f_{k}\}_{k=1}^{\infty}$ and $\{g_{k}\}_{k=1}^{\infty}$
representing the bounded invertible operator $W^{*}U$ in the basis
$\{e_{k}\}_{k=1}^{\infty}$.
\end{proof}
If the cross-Gram matrix associated to the sequences
$\{f_{k}\}_{k=1}^{\infty}$ and $\{g_{k}\}_{k=1}^{\infty}$ is
invertible and the sequence $\{f_{k}\}_{k=1}^{\infty}$ is a Bessel
sequence, then there is no need to sequence
$\{g_{k}\}_{k=1}^{\infty}$ to be a Bessel sequence, see Example
\ref{ka}. Having in mind this result, now what can we say about the
inverse of Theorem \ref{yani}? By having the assumption of the
invertibility of the cross-Gram matrix associated to
$\{f_{k}\}_{k=1}^{\infty}$ and $\{g_{k}\}_{k=1}^{\infty}$ and
completeness of the sequences $\{f_{k}\}_{k=1}^{\infty}$ and
$\{g_{k}\}_{k=1}^{\infty}$, by Example \ref{ka}, we deduce that
there is no need to sequences $\{f_{k}\}_{k=1}^{\infty}$ and
$\{g_{k}\}_{k=1}^{\infty}$ to be Riesz bases. But what can we say in
the case that $\{f_{k}\}_{k=1}^{\infty}$ and
$\{g_{k}\}_{k=1}^{\infty}$ are frames? In the following theorems we
answer to this question by considering the assumption of being frame
of both sequences.
\begin{thm}
Suppose that $\{f_{k}\}_{k=1}^{\infty}$ and
$\{g_{k}\}_{k=1}^{\infty}$ are frames for $H$ and the cross-Gram
matrix associated to these sequences is bounded and invertible. Then
$\{f_{k}\}_{k=1}^{\infty}$ and $\{g_{k}\}_{k=1}^{\infty}$ are Riesz
bases for $H$.
\end{thm}
\begin{proof}
Suppose that $G$ is the cross-Gram operator associated to $\{\langle
f_{k}, g_{j}\rangle\}_{j, k=1}^{\infty}$. So we have
$$G=T^{*}_{g_{k}}T_{f_{k}}.$$
Since $\{f_{k}\}_{k=1}^{\infty}$ and $\{g_{k}\}_{k=1}^{\infty}$ are
frames for $H$, $T_{f_{k}}$ and $T_{g_{k}}$ are bounded and
surjective operators. Now we want to show that $T_{f_{k}}$ is an
injective operator. For the given sequences
$\{c_{k}\}_{k=1}^{\infty}$, $\{b_{k}\}_{k=1}^{\infty}\in
\ell^{2}(\mathbb{N})$, suppose that
$$T_{f_{k}}\{c_{k}\}_{k=1}^{\infty}=T_{f_{k}}\{b_{k}\}_{k=1}^{\infty}.$$
Then we have
$$T^{*}_{g_{k}}T_{f_{k}}\{c_{k}\}_{k=1}^{\infty}=T^{*}_{g_{k}}T_{f_{k}}\{b_{k}\}_{k=1}^{\infty}.$$
So
$$G\{c_{k}\}_{k=1}^{\infty}=G\{b_{k}\}_{k=1}^{\infty}.$$
Since $G$ is an invertible operator, we deduce that
$\{c_{k}\}_{k=1}^{\infty}=\{b_{k}\}_{k=1}^{\infty}$ and therefore
$T_{f_{k}}$ is an injective operator and so
$\{f_{k}\}_{k=1}^{\infty}$ is a Riesz basis for $H$.
\\Now we want to show that $T_{g_{k}}$ is also a bijective operator.
Since $N(T_{g_{k}})=R(T^{*}_{g_{k}})^{\perp}$, it is enough to show
that $T^{*}_{g_{k}}: H\rightarrow\ell^{2}(\mathbb{N})$ is a
surjective operator. Since $G$ is invertible, for a given sequence
$\{c_{k}\}_{k=1}^{\infty}\in\ell^{2}(\mathbb{N})$ there exists a
sequence $\{b_{k}\}_{k=1}^{\infty}\in\ell^{2}(\mathbb{N})$ such that
$$G\{b_{k}\}_{k=1}^{\infty}=\{c_{k}\}_{k=1}^{\infty}.$$
So
$$T^{*}_{g_{k}}T_{f_{k}}\{b_{k}\}_{k=1}^{\infty}=\{c_{k}\}_{k=1}^{\infty},$$
which shows that $T^{*}_{g_{k}}$ is a surjective operator. Therefore
$\{g_{k}\}_{k=1}^{\infty}$ is a Riesz basis for $H$.
\end{proof}
\begin{cor}
Suppose that $\{f_{k}\}_{k=1}^{\infty}$ and
$\{g_{k}\}_{k=1}^{\infty}$ are frames but not Riesz bases, then $G$
can not be invertible.
\end{cor}
\begin{exam}
Suppose that $\{e_{k}\}_{k=1}^{\infty}$ is an orthonormal basis for
$H$. Consider the sequences $\{f_{k}\}_{k=1}^{\infty}=\{e_{1},
e_{1}, e_{2}, e_{3}, e_{4},...\}$ and
$\{g_{k}\}_{k=1}^{\infty}=\{e_{1}, e_{1}, e_{2}, e_{2}, e_{3},
e_{3},...\}$.  A simple calculation shows that these sequences are
frames but not Riesz bases. We get the cross-Gram matrix associated
to sequences $\{f_{k}\}_{k=1}^{\infty}$ and
$\{g_{k}\}_{k=1}^{\infty}$:
\begin{eqnarray*}
G=\left[\begin{array}{lllllll}
1&1&0&0& 0&\cdots \\\\
1&1&0&0& 0 &\cdots \\\\
0&0&1&0& 0&\cdots \\
\\
0&0&1&0&0&\cdots
\\\\
0&0&0&1&0&\cdots\\\\
0&0&0&1&0&\cdots\\\\
\vdots &\vdots &\vdots &\vdots& \vdots & \cdots \\\\
\end{array}
\right],
\end{eqnarray*}
We obtain that $\det({G})=0$ and so $G$ is not invertible.
\end{exam}
\begin{thm}\label{na}
Suppose that $\{f_{k}\}_{k=1}^{\infty}$ and
$\{g_{k}\}_{k=1}^{\infty}$ are Bessel sequences for $H$ and the
cross-Gram matrix, associated to these sequences is bounded and
invertible. Then the following statements are satisfied:
\\(i) If $\{f_{k}\}_{k=1}^{\infty}$ is a Riesz basis for $H$, then $\{g_{k}\}_{k=1}^{\infty}$ is a
Riesz basis for $H$.
\\(ii) If $\{f_{k}\}_{k=1}^{\infty}$ is a frame for $H$ and $\{g_{k}\}_{k=1}^{\infty}$ is complete in $H$, then $\{g_{k}\}_{k=1}^{\infty}$ is a
frame for $H$.
\end{thm}
\begin{proof}
(i) Suppose that $G$ is the cross-Gram operator associated to
$\{\langle f_{k}, g_{j}\rangle\}_{j, k=1}^{\infty}$. Then we have
$$G=T^{*}_{g_{k}}T_{f_{k}}.$$
Since $\{f_{k}\}_{k=1}^{\infty}$ is a Riesz basis, $T_{f_{k}}$ has a
bounded inverse and we can write
$$T^{*}_{g_{k}}=G(T_{f_{k}}^{-1}).$$
Therefore $T^{*}_{g_{k}}$ is an invertible operator and we deduce
that $\{g_{k}\}_{k=1}^{\infty}$ is a Riesz basis for $H$.
\\(ii) In order to show that $\{g_{k}\}_{k=1}^{\infty}$ is a frame for $H$ it is
enough to prove that $T_{g_{k}}$ is a surjective operator. Since
$\{g_{k}\}_{k=1}^{\infty}$ is complete in $H$, we need to show that
$T^{*}_{g_{k}}$ is injective.
\\Suppose that $$T^{*}_{g_{k}}(f_{1})=T^{*}_{g_{k}}(f_{2}),\ \ f_{1}, f_{2}\in H.$$
Since $T_{f_{k}}$ is a surjective operator, there exist sequences
$\{c_{k}\}_{k=1}^{\infty}$,
$\{b_{k}\}_{k=1}^{\infty}\in\ell^{2}(\mathbb{N})$ such that
$f_{1}=T_{f_{k}}\{c_{k}\}_{k=1}^{\infty},\ \
f_{2}=T_{f_{k}}\{b_{k}\}_{k=1}^{\infty},$ and therefore we can write
$$T^{*}_{g_{k}}T_{f_{k}}\{c_{k}\}_{k=1}^{\infty}=T^{*}_{g_{k}}T_{f_{k}}\{b_{k}\}_{k=1}^{\infty}.$$
Now by the invertibility of $G$ we deduce that
$\{c_{k}\}_{k=1}^{\infty}=\{b_{k}\}_{k=1}^{\infty},$ and so
$T_{f_{k}}\{c_{k}\}_{k=1}^{\infty}=T_{f_{k}}\{b_{k}\}_{k=1}^{\infty}.$
Hence we get the proof.
\end{proof}
By changing the role of the sequences $\{f_{k}\}_{k=1}^{\infty}$ and
$\{g_{k}\}_{k=1}^{\infty}$ in above theorem we deduce the same
results.
\begin{cor}
Suppose that $\{f_{k}\}_{k=1}^{\infty}$ is a Riesz basis and
$\{g_{k}\}_{k=1}^{\infty}$ is a Bessel sequence. Then the following
statements are satisfied:
\\(i) If $\{g_{k}\}_{k=1}^{\infty}$ is not a frame, then $G$ can not be invertible.
\\(ii) If $\{g_{k}\}_{k=1}^{\infty}$ is not NBA or NBB, then $G$ can not be
invertible.
\end{cor}
\section{Dual frames associated to cross-Gram matrix}
In this section, we investigate the cases when
$\{f_{k}\}_{k=1}^{\infty}$ and $\{g_{k}\}_{k=1}^{\infty}$ are a pair
of dual frames. Recall that for the Bessel sequences
$\{f_{k}\}_{k=1}^{\infty}$ and $\{g_{k}\}_{k=1}^{\infty}$, the pair
$(\{f_{k}\}_{k=1}^{\infty}, \{g_{k}\}_{k=1}^{\infty})$ is a dual
pair if for any $f\in H$ one of the following equivalent conditions
holds:
\\(1) $f=\sum_{k=1}^{\infty}\langle f, g_{k}\rangle f_{k}, f\in H.$
\\(2) $f=\sum_{k=1}^{\infty}\langle f, f_{k}\rangle g_{k}, f\in H.$
\\(3) $\langle f, g\rangle=\sum_{k=1}^{\infty}\langle f,
f_{k}\rangle\langle g_{k}, g\rangle,\ \ f, g\in H.$
\begin{thm}
Suppose that $\{f_{k}\}_{k=1}^{\infty}$ and
$\{g_{k}\}_{k=1}^{\infty}$ are Bessel sequences and $\langle f,
g_{k}\rangle\neq 0$ for each $k\in\mathbb{N}$ and $f\in H$. Assume
that $G$, the cross Gram operator associated to $\{\langle f_{k},
g_{j}\rangle\}_{j, k=1}^{\infty}$, is a well-defined and bounded
operator with bound $M$ such that $0<M<1$. Then
$\{f_{k}\}_{k=1}^{\infty}$ and $\{g_{k}\}_{k=1}^{\infty}$ can not be
a pair of dual frames.
\end{thm}
\begin{proof}
Suppose that $\{f_{k}\}_{k=1}^{\infty}$ is a dual frame of
$\{g_{k}\}_{k=1}^{\infty}$. Then for every $f\in H$,
\begin{equation}\label{at}
f=\sum_{k=1}^{\infty}\langle f, g_{k}\rangle f_{k}.
\end{equation}
Since $G$ is a bounded operator with bound $M$, for
$\{c_{k}\}_{k=1}^{\infty}\in\ell^{2}(\mathbb{N})$ we have
\begin{equation}\label{ali}
\sum_{j=1}^{\infty}|\sum_{k=1}^{\infty}c_{k}\langle f_{k},
g_{j}\rangle|^{2}\leq M\sum_{k=1}^{\infty}|c_{k}|^{2},
\end{equation}
Now for each $f\in H$, by (\ref{at}) and (\ref{ali}) we have
\begin{eqnarray*}
\sum_{j=1}^{\infty}|\langle f,
g_{j}\rangle|^{2}&=&\sum_{j=1}^{\infty}|\langle\sum_{k=1}^{\infty}\langle
f, g_{k}\rangle f_{k},
g_{j}\rangle|^{2}\\&=&\sum_{j=1}^{\infty}|\sum_{k=1}^{\infty}\langle
f, g_{k}\rangle\langle f_{k}, g_{j}\rangle|^{2}\leq
M\sum_{k=1}^{\infty}|\langle f, g_{k}\rangle|^{2},
\end{eqnarray*}
which is a contradiction. Therefore we get the proof.
\end{proof}
\begin{exam}
Let $\{e_{k}\}_{k=1}^{\infty}$ be an orthonormal basis for $H$.
Consider $\{f_{k}\}_{k=1}^{\infty}=\{e_{1}, e_{2},
e_{3},e_{4},...\}$ and $\{g_{k}\}_{k=1}^{\infty}=\{\frac{1}{2}e_{1},
\frac{1}{2}e_{1}, \frac{1}{3}e_{1}, \frac{1}{4}e_{1},...\}$. Suppose
that $G$ is the cross-Gram operator associated to $\{\langle f_{k},
g_{j}\rangle\}_{j, k=1}^{\infty}$. Then $G$ is a well-defined and
bounded operator with bound $\sqrt{\frac{89}{100}}$. Suppose that
$\{f_{k}\}_{k=1}^{\infty}$ and $\{g_{k}\}_{k=1}^{\infty}$ be a pair
of dual frames. Then we have
$e_{2}=\frac{1}{2}e_{1},$ which is a contradiction. Therefore
$\{f_{k}\}_{k=1}^{\infty}$ and $\{g_{k}\}_{k=1}^{\infty}$ can not be
a dual pair.
\end{exam}
\begin{thm}
Suppose that $\{f_{k}\}_{k=1}^{\infty}$ is a frame for $H$ and
$\{S^{-1}f_{k}\}_{k=1}^{\infty}$ is its canonical dual frame. Assume
that $G$ is the cross-Gram operator associated to $\{\langle f_{k},
S^{-1}f_{j}\rangle\}_{j, k=1}^{\infty}$. Then $G$ is self-adjoint
and positive operator.
\end{thm}
\begin{proof}
Since $G$ is the cross-Gram operator associated to $\{\langle f_{k},
S^{-1}f_{j}\rangle\}_{j, k=1}^{\infty}$, we have
$$G=T^{*}_{s^{-1}f_{k}}T_{f_{k}}=T^{*}_{f_{k}}S^{-1}T_{f_{k}}.$$
Therefore $G$ is self-adjoint. Now we show that $G$ is a positive
operator. For $\{c_{k}\}_{k=1}^{\infty}\in\ell^{2}(\mathbb{N})$ we
have
\begin{eqnarray*}
\langle G\{c_{k}\}_{k=1}^{\infty},
\{c_{k}\}_{k=1}^{\infty}\rangle&=&\langle
T^{*}_{f_{k}}S^{-1}T_{f_{k}}\{c_{k}\}_{k=1}^{\infty},
\{c_{k}\}_{k=1}^{\infty}\rangle\\&=&\langle
S^{-1}T_{f_{k}}\{c_{k}\}_{k=1}^{\infty},
T_{f_{k}}\{c_{k}\}_{k=1}^{\infty}\rangle\\&=&\sum_{k=1}^{\infty}|\langle
T_{f_{k}}\{c_{k}\}_{k=1}^{\infty}, S^{-1}f_{k}\rangle|^{2}.
\end{eqnarray*}
\end{proof}
\begin{cor}
If $\{f_{k}\}_{k=1}^{\infty}$ is a Riesz basis in above theorem,
then $G$ is the identity operator.
\end{cor}
\begin{exam}
Let $\{e_{k}\}_{k=1}^{\infty}$ be an orthonormal basis for $H$.
Consider $$\{f_{k}\}_{k=1}^{\infty}=\{e_{1}, e_{1}, e_{2},
e_{3},...\}.$$ The canonical dual frame is given by
$$\{S^{-1}f_{k}\}_{k=1}^{\infty}=\{\frac{1}{2}e_{1},
\frac{1}{2}e_{1}, e_{2}, e_{3},...\}.$$ The cross-Gram matrix
associated to these sequences is as follows:
\begin{eqnarray*}
G=\left[\begin{array}{llllll}
\frac{1}{2}&\frac{1}{2}&0&0 &0 & \cdots\\\\
\frac{1}{2}&\frac{1}{2}&0&0 & 0&\cdots \\\\
0&0&1&0& 0& \cdots\\
\\
0&0&0&1&0&\cdots\\
\\
\vdots& \vdots & \vdots&\vdots & \vdots & \cdots \\\\
\end{array}
\right],
\end{eqnarray*}
Then
\begin{eqnarray*}
\langle G\{c_{k}\}_{k=1}^{\infty},
\{c_{k}\}_{k=1}^{\infty}\rangle&=&\langle\{\frac{1}{2}c_{1}+\frac{1}{2}c_{2},
\frac{1}{2}c_{1}+\frac{1}{2}c_{2}, c_{3}, c_{4},...\}, \{c_{1},
c_{2}, c_{3},
c_{4},...\}\rangle\\&=&\frac{1}{2}(c_{1}+c_{2})\overline{c_{1}}+\frac{1}{2}(c_{1}+c_{2})\overline{c_{2}}+
\sum_{k=3}^{\infty}|c_{k}|^{2},
\end{eqnarray*}
which shows that $G$ is a positive operator and $G^{2}=G$.
\end{exam}
\bibliographystyle{plain}

\begin{thebibliography}{99}
\bibitem{12}
P. Balazs, Basic definition and properties of Bessel multipliers, J.
Math. Anal. Appl. {\bf 325}, 1 (2007), 571–-585.
\bibitem{B}
P. Balazs, Hilbert-Schmidt operators and frames-classification, best
approximation by multipliers and algorithms, Int. J. Wavelets
Multiresolut. Inf. Process. {\bf 6}, 2 (2008), 315-330.
\bibitem{8}
P. Balazs, B. Laback, G. Eckel, W. A. Deutsch, Time-frequency
sparsity by removing perceptually irrelevant components using a
simple model of simultaneous masking, IEEE Trans. Speech Audio
Process. {\bf 18}, 1 (2010), 34–-49.
\bibitem{9}
A. Benyi, L. Grafakos, K. Gr¨ochenig, K. Okoudjou, A class of
Fourier multipliers for modulation spaces, Appl. Comput. Harmon.
Anal. {\bf 19}, 1 (2005), 131–-139.
\bibitem{7}
O. Christensen, An Introduction to Frames and Riesz Bases,
Birkhauser, Boston, MA. 2003.
\bibitem{10}
M. H. Faroughi, E. Osgooei, A. Rahimi, $(X_{d},
X_{d}^{*})$-multipliers in Banach spaces, Banach J. Math. Anal. {\bf
7} (2013), 146-161.
\bibitem{16}
H. G. Feichtinger, G. Narimani, Fourier multipliers of classical
modulation spaces, Appl. Comput. Harmon. Anal. {\bf 21}, 3 (2006),
349–-359.
\bibitem{22}
G. F. Margrave, P. C. Gibson, J. P. Grossman, D. C. Henley, V.
Iliescu, M. P. Lamoureux, The Gabor transform, pseudodifferential
operators, and seismic deconvolution, Integr. Comput.-Aid. E. {\bf
12}, 1 (2005), 43–-55.
\bibitem{ped}
M. Pedersen, Functional Analysis in Applied Mathematics and
Engineering, CRC press, New York, 1999.
\bibitem{p}
E. Pekalska, R. P. W. Duin, The Dissimilarity Representation for
Pattern Recognition: Foundations and Applications, World Scientific
Publishing Co., Singapore, 2005.
\bibitem{Rudin}
W. Rudin, Functional Analysis, McGraw-Hill, New York, 1991.
\bibitem{11}
D. T. Stoeva, P. Balasz, Unconditional convergence and invertibility
of multipliers, arXiv:0911.2783, 2009.
\bibitem{30}
D. Wang, G. J. Brown, Computational Auditory Scene Analysis: Prin-
ciples, Algorithms, and Applications, Wiley-IEEE Press, 2006.
\end{thebibliography}

\end{document}